\documentclass[12pt]{article}
\title{Bounds on the number of Diophantine quintuples}
\author{ 
Tim Trudgian\footnote{Supported by Australian Research Council DECRA Grant DE120100173.}\\
Mathematical Sciences Institute\\ The Australian National University,
 ACT 0200, Australia\\ timothy.trudgian@anu.edu.au
}
\usepackage{url}
\usepackage{enumerate}
\usepackage{amsthm}
\usepackage{amsmath}
\usepackage{comment}
\usepackage{fullpage}
\usepackage{amssymb}
\usepackage{booktabs}
\newtheorem{goal}{Goal}
\newtheorem{thm}{Theorem}
\newtheorem{Lem}{Lemma}
\newtheorem{cor}{Corollary}
\newtheorem{conj}{Conjecture}

\begin{document}
\maketitle
\begin{abstract}
\noindent
We consider Diophantine quintuples $\{a, b, c, d, e\}$. These are sets of distinct positive integers, the product of any two elements of which is one less than a perfect square. It is conjectured that there are no Diophantine quintuples; we improve on current estimates to show that there are at most $1.9\cdot 10^{29}$ Diophantine quintuples.
\end{abstract}

\section{Introduction}
Consider the set $\{1, 3, 8, 120\}$. This has the property that the product of any two of its elements is one less than a square. Define a Diophantine $m$-tuple as a set of $m$ distinct integers $a_{1}, \ldots, a_{m}$ such that $a_{i} a_{j} +1$ is a perfect square for all $1\leq i<j\leq m$. Throughout the rest of this article we simply refer to $m$-tuples, and not to Diophantine $m$-tuples.

One may extend any triple $\{a, b, c\}$ to a quadruple $\{a, b, c, d_{+}\}$ where
\begin{equation}\label{d+}
d_{+} = a + b+ c+ 2abc + 2rst, \quad
r= \sqrt{ab +1}, \quad s = \sqrt{ac +1}, \quad t = \sqrt{bc+1},
\end{equation}
by appealing to a result by Arkin, Hoggatt and Straus \cite{AHS}. Indeed, they conjectured that \textit{every} such quadruple is formed in this way. We record this in
\begin{conj}\label{Con1}[Arkin, Hoggatt and Straus]
If $\{a, b, c, d\}$ is a quadruple then $d = d_{+}$.
\end{conj}
Note that any possible quintuple $\{a, b, c, d, e\}$ contains, \textit{inter alia} the quadruples $\{a, b, c, d\}$ and $\{a, b, c, e\}$. If Conjecture \ref{Con1} is true then $d_{+} = d = e$, whence $d$ and $e$ are not distinct. Therefore Conjecture \ref{Con1} implies

\begin{conj}\label{Con2}
There are no quintuples.
\end{conj}
Dujella \cite{Dujella2004} proved that there are finitely many quintuples. Subsequent research --- summarised in Table \ref{tab:Qbounds} --- has reduced the bound on the total number of quintuples. We prove
\begin{thm}\label{Main}
There are at most $2.32\cdot 10^{29}$ Diophantine quintuples.
\end{thm}

\begin{table}[ht]
\caption{Bounds on the number of Diophantine Quintuples}
\centering
\begin{tabular}
{cc}
\hline\hline
& Upper bound on number of quintuples\\ \hline
Dujella \cite{Dujella2004} & $10^{1930}$\\
Fujita \cite{Fujita} & $10^{276}$\\
Fillipin and Fujita \cite{FilFuj} & $10^{96}$\\
Elsholtz, Fillipin and Fujita \cite{EFF} & $6.8\cdot 10^{32}$\\
Trudgian & $1.9\cdot 10^{29}$ \\
    \hline\hline
  \end{tabular}
\label{tab:Qbounds}
\end{table}
Recent work by Wu and He \cite{WuHe} did not explicitly estimate the number of quintuples, though bounds for the second largest element $d$ were considered in some special cases --- see \S \ref{sec:Wu} for more details. We also note that the proof of Proposition 4.2 in \cite{Fujita} appears to be flawed, and hence the estimate in \cite{EFF} is too small. We repair the proof, and improve on it slightly, in \S \ref{sec:bound}.

The layout of the paper is as follows. First, in \S \ref{sec:quinces} we define several classes of quintuples and identify doubles and triples that cannot be extended to quintuples. Second, in \S\S \ref{sec:Wu} and \ref{sec:bound} we bound the size of the second largest element of a quintuple. Essential to Dujella's argument, and to all subsequent improvements, is a result by Matveev \cite{Matveev} on linear forms of logarithms. We make use of a result by Aleksentsev  \cite{Aleks} which, for our purposes, is slightly better. In several places we optimise the argument given by Fujita \cite{Fujita}.

In \S \ref{sec:sums} we estimate some sums from elementary number theory. In \S \ref{sec:calculations} we estimate the total number of quintuples, and we prove Theorem \ref{Main}.
In \S \ref{sec:quad} we define $D(-1)$-quadruples and, using one of our ancillary results, make a small improvement on the estimated number of these. In \S \ref{sec:conch} we conclude with some ideas on possible future improvements.

\subsection*{Acknowledgements}
I am grateful to Olivier Ramar\'{e} and Roger Heath-Brown for discussions about \S \ref{great}, and to Dave Platt who provided oodles of computational advice, suggestions, and banter.

\section{Triples contained within quintuples}\label{sec:quinces}
Fujita \cite{Fujita} considered three classes of triples $\{a, b, c\}$, namely
\begin{enumerate}
\item A triple of the \textit{first kind} when $c>b^{5}$.
\item A triple of the \textit{second kind} when $b>4a$ and $b^{2}\leq c \leq b^{5}$.
\item A triple of the \textit{third kind} when $b> 12 a$ and $b^{5/3} <c<b^{2}$.
\end{enumerate}
In \cite[Lem.\ 4.2]{EFF} it was shown that any quadruple contains a triple of one of the types listed above. Specifically, we have
\begin{Lem}[Lemma 4.2 in \cite{EFF}]
Let $\{a, b, c, d, e\}$ be a Diophantine quintuple with $a<b<c<d<e$. Then
\begin{enumerate}
\item
$\{a, b, d\}$ is a triple of the first kind, or
\item
\begin{enumerate}[(i)]
\item $\{a, b, d\}$  is of the second kind, with $4ab + a + b \leq c \leq b^{3/2}$, or
\item $\{a, b, d\}$  is of the second kind, with $c = a+ b + 2r$, or
\item $\{a, b, d\}$  is of the second kind, with $c> b^{3/2}$, or
\item $\{a, c, d\}$ is of the second kind, with $b<4a$ and $c = a + b + 2r$, or
\end{enumerate}
\item $\{a, c, d\}$ is of the third kind, with $b<4a$ and $c = (4ab + 2)(a + b - 2r) + 2(a+b)$.
\end{enumerate}

\end{Lem}
In \cite{PlattTrudgianQ1} it was shown that there are no quintuples $\{a, b, c, d, e\}$ such that $\{a, b, d\}$ is a triple of the first kind. 
 In \cite[\S 5]{EFF} it is shown that the \textit{number} of quintuples containing triples  is at most
 $6.74\cdot 10^{32}$ (second kind) and $1.92\cdot 10^{26}$ (third kind).
In this article we focus primarily on quintuples containing triples of the second kind.

\subsection{Doubles and Triples that need not be considered}
If a double or a triple can only be extended to a regular quadruple, it cannot be extended to a quintuple. We call such doubles and triples \textit{discards} since we do not consider them in what follows. The double $\{k, k+2\}$ \cite{Fujitak} and the triple $\{F_{2k}, F_{2k+2}, F_{2k+4}\}$ \cite{DujellaFib} are discards for $k\geq 1$, where $F_{n}$ denotes the $n$th Fibonacci number.

Kedlaya \cite{Kedlaya} has shown that the following are discards:
\begin{equation}\label{kettle}
\{1, 8, 15\}, \{1, 8, 120\}, \{1, 15, 24\}, \{1, 24, 35\}, \{2, 12, 24\}.
\end{equation}

He and Togb\'{e} \cite{HeTogbePeriod} proved that $\{k+1, 4k, 9k+3\}$ is a discard for any $k\geq 1$. In \cite{HeTogbe1} they proved that 
\begin{equation}\label{toggle}
\{k, A^{2}k + 2A, (A+1)^{2}k + 2(A+1)\}, \quad(k\geq 1),
\end{equation}
 is a discard for all $3 \leq A \leq 10$. When $A=1$, the triple in (\ref{toggle}) is covered by Fujita's double $\{k, k+2\}$; when $A=2$, He and Togb\'{e} remark \cite[p.\ 101]{HeTogbe1} that one can use a method similar to that in \cite{HeTogbePeriod} to prove that the triple is a discard. Finally, in \cite{HeTogbe2} they prove that (\ref{toggle}) is a discard for all $A\geq 52330$.

Filipin, Fujita and Togb\'{e} \cite[Cor.\ 1.6, 1.9]{FFT} proved that the following are all discards:
\begin{equation}\label{neck1}
\{k^{2} -1, k^{2} + 2k\}, \quad\quad \{2k^{2} - 2k, 2k^{2} + 2k\}, \quad\quad \{k, 4k -4\} \quad (k\geq 2),
\end{equation}
and
\begin{equation}\label{neck2}
\{3k^{2} - 2k, 3k^{2} + 4k + 1\}, \quad\quad \{3k^{2} + 2k, 3k^{2} + 8k + 5\}, \quad\quad \{k, 4k +4\} \quad (k\geq 1).
\end{equation}

We use the above discards to exhibit the smallest possible values of $b$ in triples of the kinds 2(i), 2(ii) and 2(iii). A quick computer search establishes that there are no $b< 1680$ such that $\{a, b, d\}$ is a triple of the kind 2(i).
Indeed, the only such quadruples $\{a, b, c, d\}$ with $b\leq 10000$ are
\begin{equation*}\label{tablet}
\begin{split}
&\{1, 1680, 23408, 157351935\}, \{1, 4095, 139128, 2279203080\},\\
& \{3, 1680, 23408, 471955461\}, \{8, 4095, 139128, 18231619581\}.
\end{split}
\end{equation*}

Using (\ref{toggle}) with $A=2$, and (\ref{neck1}), we see that the second-smallest element inside a triple of the kind 2(ii) is 21: this corresponds to the quadruple $\{3, 21, 40, 10208\}$. Finally, using (\ref{kettle}), (\ref{neck1}) and (\ref{neck2}) we see that the second-smallest element inside a triple of the kind 2(iii) is $b=15$: this corresponds to the two quadruples $\{1, 15, 528, 32760\}$ and $\{1, 15, 1520, 94248\}$. We record all of these results in
\begin{Lem}\label{lem:jill}
Let $\{A, B, C\}$ be a triple of the kind 2(i), 2(ii), or 2(iii). Then
\begin{equation*}\label{pans}
\textrm{2(i)}\quad B\geq 1680, C\geq 10^{8}, \quad \textrm{2(ii)}\quad B\geq 21, C\geq 10208, \quad \textrm{2(iii)}\quad B\geq 15, C\geq 32760.
\end{equation*}
\end{Lem}
We conclude this section citing a result by Fujita \cite{fujita2009any}: any potential quintuple $\{a, b, c, d, e\}$ must have $d= d_{+}$, where $d_{+}$ is given in (\ref{d+}). This shows that $d \geq 4abc$ --- we shall use this result frequently in \S \ref{sec:calculations}.

\section{Wu and He's argument}\label{sec:Wu}
The values in a quadruple are linked by a series of Pellian equations. These equations have solutions $v_{m}$ and $w_{n}$ with non-negative integral parameters $m$ and $n$. Wu and He \cite[Lem.\ 2]{WuHe} give a refined version of Lemma 3 in \cite{Dujella2004} by proving
\begin{Lem}[Wu and He]
If $B\geq 8$ and $v_{2m} = w_{2n}$ has solutions for $m\geq 3, n\geq 2$, then $m>0.48 B^{-1/2}C^{1/2}$.
\end{Lem}
It appears as though their proof is incomplete. We give the following, tailored version below, with a small improvement.
\begin{Lem}\label{rib}
Let $\{A, B, C\}$ be one of the triples 2(i), 2(ii), 2(iii). Then, for $B\geq 8$, if $v_{2m} = w_{2n}$ has solutions for $m\geq 3, n\geq 2$, then the following bounds for $m$ hold
\begin{equation*}\label{pots}
\textrm{2(i)}\quad m\geq 1.3330 C^{1/4}, \quad \textrm{2(ii)}\quad m\geq 0.9282 C^{1/4}, \quad \textrm{2(iii)}\quad m\geq 0.8609 C^{3/10}.
\end{equation*}

\end{Lem}
\begin{proof}
We proceed as in \cite{WuHe}. Assume that $m\leq \alpha B^{-1/2} C^{1/2}$ for some $\alpha$ to be determined later. By Lemma 4 in \cite{Dujella2001} we have
\begin{equation}\label{rump}
Am^{2} + \lambda Sm \equiv B n^{2} + \lambda Tn \pmod{4C},
\end{equation}
for some $\lambda = \pm 1$. We aim at showing that (\ref{rump}) is actually an equality. Since $B\geq 8$ it is easy to see that each of the four terms in (\ref{rump}) is less than $C$. We conclude that $Am^{2} - Bn^{2} = \lambda(Tn - Sm)$, which, upon rearranging, and invoking the definitions of $T$ and $S$ gives
\begin{equation}\label{sirloin}
m^{2} - n^{2} = (C + \lambda(Tn + Sm))(Bn^{2} - Am^{2}).
\end{equation}
We now aim at showing that $Bn^{2} \neq Am^{2}$, so that, since the second term on the right of (\ref{sirloin}) is an integer, we have
\begin{equation}\label{shoulder}
|m^{2} - n^{2}| \geq |C + \lambda(Tn + Sm)|.
\end{equation}
We assume, in order to obtain a contradiction, that $m = \sqrt{B/A} n$. Since $n \leq m \leq 2n$ we certainly obtain a contradiction when $\{A, B, C\}$ is a triple of the second kind, since then $B/A>4$ whence $m>2n$. It is unclear how Wu and He derive a contradiction in general.

There are two choices for $\lambda$ in (\ref{shoulder}). The choice $\lambda = 1$ shows that $m^{2}>C$, which contradicts our upper bound on $m$. The choice $\lambda = -1$ forces us to consider
\begin{equation}\label{neck}
m^{2} - n^{2} \geq C - (Tn + Sm),
\end{equation}
since $Tn + Sm < 2Tn <C$. Rearranging (\ref{neck}) and using the bound $n\geq m/2$ we obtain
\begin{equation}\label{chop}
C\leq m(T+S) + \frac{3}{4} m^{2} \leq m(BC)^{1/2} \left( \sqrt{1 + 1/(BC)} + \sqrt{\frac{1}{4} + 1/(BC)}\right) + \frac{3}{4} m^{2}.
\end{equation}
We insert our upper bound for $m$ and the lower bound for $B$ and $C$ from Lemma \ref{lem:jill} in (\ref{chop}). This yields an expression in terms of $\alpha$: we wish to choose the largest $\alpha$ for which this is less than $C$, which yields a contradiction. Having solved for $\alpha$ we note that, for
 $\{a, b, d\} = \{A, B, C\}$ a triple of the second kind, we have
\begin{equation}\label{parse}
\begin{split}
&\textrm{2(i)} \qquad C = d > 4abc = 4ABc \geq 16 B^{2} \Rightarrow B < \frac{C^{1/2}}{4}\\
&\textrm{2(ii)} \qquad C = d > 4abc = 4ABc  \geq 4 B^{2} \Rightarrow B < \frac{C^{1/2}}{2}\\
&\textrm{2(iii)} \qquad C = d > 4abc = 4ABc > 4B^{5/2} \Rightarrow B < \left(\frac{C}{4}\right)^{2/5}.
\end{split}
\end{equation}
Using the bounds in (\ref{parse}) and the values of $\alpha$ proves the lemma.
\end{proof}

We now proceed to correcting and improving on a result by Fujita.

\section{Improving Proposition 4.2 in \cite{Fujita}}\label{sec:bound}
We give two estimates of Proposition 4.2 in \cite{Fujita}. The first holds in general; the second uses a slight improvement for quintuples containing triples of the second kind. In \S \ref{sec:calculations} we use the second version and Lemma \ref{rib} to reduce the bound on the number of Diophantine quintuples.

First we use quote a result of Aleksentsev.

\begin{thm}\label{Alex}
Let $\Lambda$ be a linear form in logarithms of $n$ multiplicatively independent totally real algebraic numbers $\alpha_{1}, \ldots \alpha_{n}$, with rational coefficients $b_{1}, \ldots, b_{n}$. Let $h(\alpha_{j})$ denote the absolute logarithmic height of $\alpha_{j}$ for $1\leq j \leq n$. Let $d$ be the degree of the number field $\mathcal{K} = \mathcal{Q}(\alpha_{1}, \ldots, \alpha_{n})$, and let $A_{j} = \max(d h(\alpha_{j}), |\log \alpha_{j}|, 1)$. Finally, let
\begin{equation}\label{tail}
E = \max\left( \max_{1 \leq i, j \leq n} \left\{ \frac{|b_{i}|}{A_{j}} + \frac{|b_{j}|}{A_{i}}\right\}, 3\right).
\end{equation}
Then
\begin{equation}\label{kidney}
\log |\Lambda| \geq - 5.3 e n^{1/2} (n+1)(n+8)^{2}(n+5)(31.5)^{n} d^{2} (\log E) A_{1}\cdots A_{n} \log(3nd).
\end{equation}
\end{thm}
We have made use of the first displayed equation on \cite[p.\ 2]{Aleks} to define $E$ in (\ref{tail}) as this makes our application easier. We note also that, while Aleksentsev's result is worse than Matveev's for large $n$, when $n=3$ we obtain a slight improvement. We apply Theorem \ref{Alex} for $d=4, n=3$ and to 
\begin{equation*}\label{cheek}
\alpha_{1} = S+ \sqrt{AC}, \quad \alpha_{2} = T + \sqrt{BC}, \quad \alpha_{3} = \frac{\sqrt{B}(\sqrt{C} \pm \sqrt{A})}{\sqrt{A} \left( \sqrt{C} \pm \sqrt{B}\right)}, \quad \Lambda = j\alpha_{1} - k\alpha_{2} + \alpha_{3}.
\end{equation*}
We proceed, as in Fujita \cite[\S 4]{Fujita} and Dujella \cite[\S 8]{Dujella2004}.
Our starting point uses Lemma 3.3 in \cite{Fujita} which states that any quadruple contains a standard triple. We denote this triple by $\{A, B, C\}$, the quadruple by $\{a, b, c, d\}$, and the attached sequences by $\{V_{j}\}$ and $\{W_{k}\}$, where $k\leq j$. We also set $S = \sqrt{AC +1}$ and $T = \sqrt{BC +1}$. 

Let $C\geq C_{0}$. In general we have $C\geq B^{5/3}$; in the case of a triple of the second kind, in which we are most interested, we have $C\geq B^{2}$ and $B> 4A$. We provide details of the general case.
To apply Theorem \ref{Alex} we first estimate
\begin{equation}\label{eq:A1}
A_{1} = 2 \log \alpha_{1} \leq 2 \log\left(2 \sqrt{AC +1}\right).
\end{equation}
Since $A<B<C^{3/5}$, we may bound the right hand side of (\ref{eq:A1}) to show that
\begin{equation}\label{eq:A1bound}
A_{1} \leq g_{1}(C_{0}) \log C,
\end{equation}
where
\begin{equation*}\label{eq:g1}
g_{1}(x) = \frac{8}{5} + \frac{2 \log 2 + \log (1+ x^{-8/5})}{\log x}.
\end{equation*}
The same bound holds for $A_{2} = 2 \log \alpha_{2}$.
We obtain lower bounds for $A_{1}$ and $A_{2}$ in a similar fashion. Since $A_{1} \geq 2 \log \left( 2 \sqrt{AC}\right)$ we have
\begin{equation}\label{eq:A1lower}
A_{1} \geq g_{2}(C_{1}, A_{0}) \log C,
\end{equation}
where 
\begin{equation*}\label{eq:g2}
g_{2}(C_{1}, A_{0}) = 1 + \frac{2 \log 2 + \log A_{0}}{\log C_{1}}, \quad (A\geq A_{0}, C\leq C_{1}).
\end{equation*}

We have $A_{3} = 4 h(\alpha_{3})$, where the leading coefficient of $\alpha_{3}$ is $a_{0} = A^{2}(C-B)^{2}$. It is easy to show that 
\begin{equation}\label{mugs}
\sqrt{\frac{B}{A}} < \frac{\sqrt{B}(\sqrt{C} \pm \sqrt{A})}{\sqrt{A}(\sqrt{C} - \sqrt{B})} \leq \frac{\sqrt{B}(\sqrt{C} + \sqrt{B})^{2}}{\sqrt{A}(C - B)} \leq \frac{\sqrt{B}(1 + \sqrt{B/C})^{2}}{\sqrt{A}(1 - B/C)}.
\end{equation}
Since the function
\begin{equation*}\label{eq:g4}
g_{3}(x) = \frac{(1+ \sqrt{x})^{2} }{1-x}
\end{equation*}
is increasing for $x\in(0, 1)$ we have
\begin{equation*}\label{eq:A3bound}
A_{3} \leq \log \left\{ g_{3}(C_{0}^{-2/5}) C^{16/5}\right\}, \quad (C_{0}\leq C \leq C_{1}).
\end{equation*}
We therefore obtain
\begin{equation}\label{eq:A3actual}
A_{3} \leq g_{4}(C_{0}) \log C,
\end{equation}
where
\begin{equation*}\label{eq:g5}
g_{4}(C_{0}) = \frac{16}{5} + \frac{ \log g_{3}(C_{0}^{-2/5})}{\log C_{0}}.
\end{equation*}
Also, since $A_{3} \geq \log \left( A^{2} (C-B)^{2}\right)$ we have
\begin{equation}\label{eq:A3lower}
A_{3} \geq g_{5}(A_{0}, C_{0}) \log C,
\end{equation}
where
\begin{equation*}\label{eq:g6}
g_{5}(A_{0}, C_{0}) = 2 + \frac{2\log A_{0} + 2 \log(1- C_{0}^{-2/5})}{\log C_{0}}.
\end{equation*}
Using the fact that $g_{1}(x)/g_{5}(1, x)$ is decreasing in $x$ we find that $E \leq 2j/(g_{2}\log C_{0})$.

We are now in a position to evaluate the terms on the right side of (\ref{kidney}). By (\ref{eq:A1bound}), (\ref{eq:A1lower}) and (\ref{eq:A3actual}) we have
\begin{equation}\label{eq:miracle}
-\log \Lambda \leq 1.7315\cdot 10^{11} g_{1} g_{4} \log  \left( \frac{2j}{g_{2} \log C_{0}}\right) \log^{3} C.
\end{equation}
We can bound the left side of (\ref{eq:miracle}) using \cite[(4.1)]{Fujita}, which states that
\begin{equation*}\label{lamb}
0<\Lambda < \frac{8}{3} AC \xi^{-2j}.
\end{equation*}
Since $\xi = \sqrt{AC +1} + \sqrt{AC} \geq 2 \sqrt{AC}$  we have
\begin{equation}\label{eq:lam}
\log \Lambda \leq - g_{6} j \log C,
\end{equation}
where
\begin{equation}\label{eq:g7}
g_{6}(A_{0}, C_{0}, C_{1}, j_{0}) = 1 + \frac{2 \log (2 \sqrt{A_{0}})}{\log C_{1}} - \frac{2}{j_{0}} - \frac{ \log \frac{8}{3}}{j_{0} \log C_{0}}, \quad (C_{0} \leq C \leq C_{1}, A\geq A_{0}, j\geq j_{0}).
\end{equation}
If all that is known is that $j\geq j_{0} = 10^{10}$, say, then (\ref{eq:g7}) gives $g_{6} \geq 0.9999$. Fujita \cite{Fujita},  six lines from the bottom of page 23, appears to have substituted a lower bound for an upper bound in this calculation, leading to $g_{6} \geq 1.2006$. This appears to be an error; the corresponding bound in Dujella \cite{Dujella2004}, on the sixth line of page 23, is correct.

Combining (\ref{eq:lam}) and (\ref{eq:miracle}) proves\begin{Lem}\label{lem:gen}
\begin{equation*}\label{eq:tot_b}
\frac{j}{\log \left( \frac{2j}{g_{2} \log C_{0}}\right)} \leq 1.7315\cdot 10^{11} g_{1}^{2} g_{4} g_{6}^{-1} \log^{2} C.
\end{equation*}
\end{Lem}
A particular case of Lemma \ref{lem:gen} follows from the bounds $(A_{0}, B_{0}, C_{0}) = (1, 8,6440)$
\begin{cor}\label{Dover}
Suppose that $\{a, b, c, d, e\}$ is a quintuple with $j\geq 10^{10}$. Then
\begin{equation*}\label{snooker}
\frac{j}{\log 0.228 j} \leq 1.7548\cdot 10^{12}\log^{2} C.
\end{equation*}
\end{cor}
This repairs Proposition 4.2 in \cite{Fujita} (in fact, it is a slight improvement) and reinstates the results in \cite{EFF} that are contingent upon such a bound. 

We can run the same argument, this time tailored to triples $\{A, B, C\}$ of the second kind, namely, those with $A< B/4 < C^{1/2}/4$. This leads to bounds of the form (\ref{eq:A1bound}), (\ref{eq:A1lower}), (\ref{eq:A3actual}), (\ref{eq:A3lower}), and (\ref{eq:g7}) but with a slightly different function $h_{i}$ in place of the $g_{i}$ for $i=1,3, 4$. We only give details for the modification of $g_{3}(x)$. We may substitute $\sqrt{B}(\sqrt{C} + z \sqrt{B})$ for the numerator in the third fraction in (\ref{mugs}), and since $B>4A$ we can solve for $z = \frac{3}{4}$. The other modifications are straightforward.  We present the results in

\begin{Lem}\label{lem:spec}
Let $\{a, b, c, d, e\}$ be a quintuple containing a triple of the second kind. We have
\begin{equation}\label{eq:special}
\frac{j}{\log \left( \frac{2j}{g_{2} \log C_{0}}\right)} \leq 1.7315\cdot 10^{11} h_{1}^{2} h_{4} g_{6}^{-1} \log^{2} C.
\end{equation}
where
\begin{equation*}
\begin{split}
h_{1}(C_{0}, C_{1}) &= \frac{3}{2} + \frac{\log (1+ 4/C_{0}^{3/2}) - \log 4}{\log C_{1}}\\
h_{3}(x)& =  \frac{(1+ \frac{3}{4}\sqrt{x})^{2} }{1-x}\\
h_{4}(C_{0}) & = 3 + \frac{2 \log h_{3}(C_{0}^{-1/2})}{\log C_{0}}. \\
\end{split}
\end{equation*}
\end{Lem}
As before, using specific bounds of $(A_{0}, B_{0}, C_{0}) = (1, 8, 6440)$ we have
\begin{cor}\label{Dover2}
Suppose that $\{a, b, c, d, e\}$ is a quintuple containing a triple $\{A, B, C\}$ of the second kind, and that $j\geq 10^{10}$. Hence
\begin{equation*}\label{9ball}
\frac{j}{\log 0.228 j} \leq 1.162\cdot 10^{12}\log^{2} C.
\end{equation*}
\end{cor}
We now use Lemma \ref{rib} and set $j=2m$ in (\ref{eq:special}). Starting with the bound $ d<C_{1} = 4.2\cdot10^{76}$, we apply (\ref{eq:special}) to obtain a new upper bound on $d$. We iterate this procedure to obtain
\begin{thm}\label{tim:2}
Suppose that $\{a, b, c, d, e\}$ is a Diophantine quintuple containing a triple of the kind 2(i), 2(ii), or 2(iii). Then we have
\begin{equation*}
2(i) \quad d\leq 4.02\cdot 10^{70}, \quad 2(ii) \quad d\leq 2.09\cdot 10^{71}, \quad 2(iii) \quad d\leq 9.12\cdot 10^{58}.
\end{equation*}
\end{thm} 
We have not pursued a tailored version of Corollary \ref{Dover} when the quintuple contains a triple of the third kind, though this is certainly possible.
We now change gears, and aim at converting a bound on $d$ in Theorem \ref{tim:2} into a bound on the \textit{number} of quintuples.

\section{Some number-theoretic sums}\label{sec:sums}

In \cite{EFF} the following results are proved
\begin{Lem}[Lemma 3.1 in \cite{EFF}]\label{EFF3.1}
If $N\geq 3$ then $\sum_{n=1}^{N} 2^{\omega(n)} < N(\log N + 1)$.
\end{Lem}

\begin{Lem}[Lemma 3.3 in \cite{EFF}]\label{EFF3.3}
If $N\geq 1$ then $\sum_{n=1}^{N} 4^{\omega(n)} < \frac{N}{6}(\log N + 2)^{3}$.
\end{Lem}
\begin{Lem}[Lemma 3.4 in \cite{EFF}]\label{Vine}
\ 
\begin{enumerate}
\item
The number of solutions of $x^{2} \equiv 1 \pmod b$ with $0<x<b$ is at most $2^{\omega(b) +1}$.
\item The number of solutions of $x^{2} \equiv -1 \pmod b$ with $0<x<b$ is at most $2^{\omega(b)}$.
\end{enumerate}
\end{Lem}

\begin{Lem}[Lemma 3.5 in \cite{EFF}]\label{Lem9}
For $N\geq 2$, we have
\begin{equation*}\label{eq:high}
\sum_{n=2}^{N} d(n^{2} -1) < 2N \left\{ \log^{2}N + 4 \log N+ 2\right\}.
\end{equation*}
\end{Lem}

\begin{Lem}[Lemma 3.6 \cite{EFF}]\label{Lem10}
Let $d(n, A)$ count those divisors of $n$ not exceeding $A$. For all $N\geq 1$ and $A\geq 1$ we have
\begin{equation*}\label{duck}
\sum_{n=1}^{N} d(n^{2} -1, A) \leq 2N\left( \log^{2} A + 4\log A + 2\right).
\end{equation*}
\end{Lem}

\begin{Lem}[Lemma 3.7 \cite{EFF}]\label{Lem10a}
For $N\geq 2$, we have
\begin{equation*}\label{duck2}
\sum_{n=1}^{N} d(n^{2} +1) \leq N\left( \log^{2} N + 4\log N + 2\right).
\end{equation*}
\end{Lem}

The goal of this section is to prove Lemma \ref{bourbon}, Lemma \ref{Lem14}, Corollary \ref{Core3}, and Lemma \ref{Lem15}. These improve on Lemma \ref{EFF3.1}, Lemma \ref{Lem9}, Lemma \ref{Lem10} and Lemma \ref{Lem10a}, respectively.

\begin{Lem}\label{bourbon}
 For all $x\geq 1$ we have
 \begin{equation}\label{actual1}
  \sum_{n\leq x} \frac{2^{\omega(n)}}{n} \leq 3 \pi^{-2} \log^{2} x + 1.3949\log x + 0.4107+ 3.253x^{-1/3}.
  \end{equation}
  For $x>1$ we have
 \begin{equation*}\label{actual2}
 \sum_{n\leq x} 2^{\omega(n)} \leq 6 \pi^{-2} x \log x + 0.787x  - 0.3762 + 8.14 x^{2/3}.
 \end{equation*}
 \end{Lem}
Our approach is to use the following results obtained by Berkane, Bordell\`{e}s and Ramar\'{e}\footnote{We note that Corollary 1.8 in \cite{BBR} contains a small misprint: we have corrected $-\gamma_{1}$ to $-2\gamma_{1}$ in (\ref{tau}).}. In what follows we write $f(x) = \vartheta(g(x))$ if $|f(x)|\leq g(x)$ for all $x$ under consideration.
\begin{Lem}[Cor.\ 1.8 in \cite{BBR}]\label{Peter}
For all $t>0$
\begin{equation}\label{tau}
\sum_{n\leq t} \frac{d(n)}{n} = \frac{1}{2} \log ^{2} t + 2 \gamma \log t + \gamma^{2} - 2\gamma_{1} + \vartheta(1.16 t^{-1/3}),
\end{equation}
where $\gamma$ is Euler's constant and $\gamma_{1}$ is the second Stieltjes constant, which satisfies $-0.07282<\gamma_{1} < -0.07281.$
\end{Lem}

\begin{Lem}[Lem.\ 3.2 \cite{RamSch}]\label{jeep}
Let $\{g_{n}\}_{n\geq 1}, \{h_{n}\}_{n\geq 1}$ and $\{k_{n}\}_{n\geq 1}$ be three sequences of complex numbers satisfying $ g = h* k$. Let $H(s) = \sum_{n\geq 1} h_{n} n^{-s}$ and $H^{*}(s) = \sum_{n\geq 1} |h_{n}| n^{-s}$, where $H^{*}(s)$ converges for $\Re(s)\geq -\frac{1}{3}$. If there are four constants $A, B, C$ and $D$ satisfying
\begin{equation*}\label{whisky}
\sum_{n\leq t} k_{n} = A\log^{2} t + B \log t +  C + \vartheta(D t^{-1/3}), \quad (t>0),
\end{equation*}
then
\begin{equation*}\label{tot}
\sum_{n \leq t} g_{n} = u \log^{2} t + v \log t + w + \vartheta(D t^{-1/3} H^{*}(-1/3)),
\end{equation*}
and
\begin{equation*}\label{rum}
\sum_{n\leq t} n g_{n} = Ut \log t + Vt +  W + \vartheta(2.5 D t^{2/3} H^{*}(-1/3)),
\end{equation*}
where
\begin{equation*}\label{gin}
\begin{split}
u&= AH(0), \quad v = 2AH'(0) + BH(0), \quad w =AH''(0) + BH'(0) + CH(0)\\
U &= 2AH(0), \quad V = -2A H(0) + 2A H'(0) + B H(0),\\
W&= A(H''(0) - 2H'(0) + 2 H(0)) + B(H'(0) - H(0)) + CH(0).
\end{split}
\end{equation*}
\end{Lem}
We note that, on page 11 of \cite{RamSch} there appears to be a misprint in the value given for $U$: we have corrected this in Lemma \ref{jeep}.

\begin{proof}[Proof of Lemma \ref{bourbon}]
Choosing $g_{n} = 2^{\omega(n)}/n$ and $k_{n} = d(n)/n$ we have $H(s) = \zeta(2(s+1))^{-1}$. Therefore, for all $\sigma > -\frac{1}{2}$ we have
\begin{equation*}\label{port}
H^{*}(s) = \prod_{p} (1 + p^{-2(s+1)})=\frac{\zeta(2(s + 1))}{\zeta(4(s + 1))},
\end{equation*}
which converges for all $\sigma > -\frac{1}{2}$. Therefore we may apply Lemma \ref{jeep} with
\begin{equation*}
H(0) = 6\pi^{-2}, \quad H'(0) = -72 \zeta'(2) \pi^{-4}, \quad H''(0) = 1728 \zeta'(2)^{2} \pi^{-6} - 144 \zeta''(2) \pi^{-4}.
\end{equation*}
 From (\ref{tau}) we have that $A= \frac{1}{2}, B = 2\gamma, C = \gamma^{2} - 2\gamma_{1}$ and $D=1.16$. This gives us the lemma with exact values, and, in fact, gives upper and lower bounds for the sum. Inserting numerical values for $\zeta'(2)$ and $\zeta''(2)$ and taking care of rounding proves the lemma.
 \end{proof}
Inserting (\ref{actual1}) into the proof of Lemma 3.6 in \cite{EFF} gives the following improvement on Lemma \ref{Lem10}.
\begin{cor}\label{Core3}
Let $d(n, A)$ count those divisors of $n$ not exceeding $A$. For all $N\geq 1$ and $A >1$ we have
\begin{equation}\label{duck3}
\sum_{n=1}^{N} d(n^{2} -1, A) \leq 4N\left( 3\pi^{-2} \log^{2} A + 1.3949\log A + 0.4107 + 3.253A^{-1/3}\right).
\end{equation}
\end{cor}

The asymptotic order of (\ref{duck3}) is unclear --- similar sums are dealt with in  \cite{KAB:restricted}. 

We now turn to Lemma \ref{Lem9}: this is the correct order of magnitude since Hooley \cite[p.\ 97]{Hooley} showed that $\sum_{n=2}^{N} d(n^{2} -1) \sim c N \log^{2} N$ for some positive constant $c$. Since $d(n^{2}-1) = d((n+1)(n-1)) \leq d(n+1)d(n-1)$ one could use the following result by Ingham \cite[(1)]{Ingham1927}, namely, that
\begin{equation}\label{eq:Ing}
\sum_{n=1}^{N-1} d(n) d(n+2) \sim \frac{9}{\pi^{2}} x \log^{2} x.
\end{equation}

One could certainly make (\ref{eq:Ing}) explicit, though we do not pursue this here.  However, we can make a small improvement in
\begin{Lem}\label{Lem14}
\begin{equation*}\label{eq:low}
\sum_{n=2}^{N} d(n^{2} -1) \leq 4N \left(3 \pi^{-2} \log^{2} N + 1.3949\log N + 0.4107 + 3.253N^{-1/3}\right).
\end{equation*}
\end{Lem}
\begin{proof}
We proceed as in the proof of \cite[Lem.\ 3.5]{EFF}. Note that
\begin{equation*}
\sum_{n=2}^{N} d(n^{2} - 1) \leq 2 \sum_{n=1}^{N} \sum_{y|n^{2} -1, y=1}^{n} 1 = 2\sum_{y=1}^{N} \sum_{n=y, n^{2} \equiv 1 \pmod y}^{N} 1 \leq 4N \sum_{y=1}^{N} \frac{2^{\omega(y)}}{y},
\end{equation*}
where we have used Lemma \ref{Vine}. The result follows upon using the first part of Lemma \ref{bourbon} to bound the final sum.
\end{proof}
Similarly, we improve on Lemma \ref{Lem10a} in 
\begin{Lem}\label{Lem15}
For $N\geq 2$ we have
\begin{equation*}\label{eq:low2}
\sum_{n=2}^{N} d(n^{2} +1) < 2N \left(3 \pi^{-2} \log^{2} N + 1.3949\log N + 0.4107 + 3.253N^{-1/3}\right).
\end{equation*}
\end{Lem}

Although we cannot give an explicit improvement on Lemma \ref{EFF3.3}, we do calculate an asymptotic formula for the sum. 

\subsection{The sum $\sum_{n\leq x} 4^{\omega(n)}$}\label{great}
We proceed to bound $\sum_{n\leq x} 4^{\omega(n)}/n$, thereafter obtaining our desired bound via partial summation. We write  $4^{\omega(n)}/n = d(n)/n * \left(d(n)/n * h\right)$ and proceed to determine the function $h$. We find that $h(n)$ is a multiplicative function completely determined by 
\begin{equation*}\label{vodka}
\begin{split}
h(1) &= 1, \quad h(p) = 0, \quad h(p^{2}) = -6p^{-2},\\
h(p^{3}) &= 8p^{-3}, \quad h(p^{4}) = -3 p^{-4}, \quad h(p^{e}) = 0, \quad (e\geq 5).
\end{split}
\end{equation*}
We therefore have
\begin{equation*}
\begin{split}
H(s) &= \prod_{p} (1 - 6p^{-2(s+1)} + 8 p^{-3(s+1)} - 3 p^{-4(s+1)})\\
H^{*}(s) &=\prod_{p} (1 + 6p^{-2(s+1)} + 8 p^{-3(s+1)} +3 p^{-4(s+1)}),
\end{split}
\end{equation*}
both of which are convergent for $\sigma > -\frac{1}{2}$. We let $g = d(n)/n *h$ and apply the Dirichlet hyperbola method to find that
\begin{equation*}\label{base}
\begin{split}
\sum_{n\leq X} \frac{4^{\omega(n)}}{n} &= \sum_{a \leq \sqrt{X}} \frac{d(a)}{a} \sum_{b\leq \frac{X}{a}} g(b) + \sum_{b\leq \sqrt{X}} g(b) \sum_{a\leq \frac{X}{b}} \frac{d(a)}{a} - \sum_{a\leq \sqrt{X}} \frac{d(a)}{a} \sum_{b\leq \sqrt{X}} g(b)\\
&= \sum_{1} + \sum_{2} - \sum_{3},
\end{split}
\end{equation*}
say. We may use Lemmas \ref{Peter} and \ref{jeep} to bound the sum of the $g(n)$s. We find that
\begin{equation}\label{pontifex}
\sum_{n\leq x} 4^{\omega(n)} = \frac{1}{6} H(0) x \log^{3} x + \left((2\gamma - \frac{1}{2}) H(0) + \frac{H'(0)}{2}\right) x\log^{2} x + O(x\log x).
\end{equation}
However, we run into difficulties in the lower order terms. This suggests that our approach of writing $4^{\omega(n)}/n = d(n)/n * d(n)/n * h$ needs to be altered to obtain a completely explicit version of (\ref{pontifex}). Provided the lower order terms can be tamed, given that $H(0) = 0.1148\ldots$, one expects a bound in (\ref{pontifex}) to improve on that in \cite{EFF} by almost one order of magnitude. 

\section{Calculation of the number of quintuples}\label{sec:calculations}
We have kept this section as brief as possible, merely showing how one can insert our refined values into the  proofs given in \cite{EFF}.

\subsection{Triples of the kind 2(i)}
By Lemma \ref{lem:jill} we have $r= \sqrt{ab +1}> \sqrt{1681}=41$, whence
\begin{equation*}\label{ming}
d> 4abc> 16a^{2} b^{2} > 16 r^{4} \left(1-\frac{1}{41}^{2}\right)^{2}.
\end{equation*}
It follows from Theorem \ref{tim:2} that $r< 2.24\cdot 10^{17}$. Since $r^{2} -1 = ab$, it follows from Lemma \ref{Lem14} that there are at most $2.43\cdot 10^{20}$ pairs $(a, b)$ with $a<b$. Also, since $d> 4abc> 20b^{2}$ we have $b< 4.49\cdot 10^{34}$. Since the product of the first $25$ primes exceeds $10^{36}$ we conclude that $\omega(b) \leq 24$. Inserting this into the proof in \cite{EFF} we that the number of quintuples is at most
\begin{equation}\label{quince1}
2.24\cdot 10^{17} \cdot 3 \cdot 4 \cdot 2^{26}\leq 1.81\cdot 10^{29}.
\end{equation}

\subsection{Triples of the kind 2(ii)}
We have $r< \left(d/12\right)^{1/3}$ so that, by Theorem \ref{tim:2} we have $r< 2.6\cdot 10^{23}$. Using Lemma \ref{Lem14} and following the proof in \cite{EFF} we find that the number of quintuples is at most
\begin{equation}\label{quince2}
2.0\cdot 10^{27}.
\end{equation}

\subsection{Triples of the kind 2(iii)}
Let $\eta$ be a parameter: we consider the cases $a> \eta$ and $a\leq \eta$ and optimise over $\eta$. In the former, we have $d> 4abc > 4\eta b^{5/2}$ so that $b< (d/(4\eta))^{2/5}:= N_{3a}$. Hence, by Lemma \ref{EFF3.3} and the argument in \cite{EFF}, the number of quintuples is at most
\begin{equation}\label{quince3a}
\frac{N_{3a}}{6} \left( \log N_{3a} + 2)^{3}\right) \cdot 8 \cdot 5 \cdot 4.
\end{equation}

When $a< \eta$, we have $b< (d/(4a))^{2/5}$ so that $r^{2} = ab + 1 < a(d/(4a))^{2/5} +1$. Thus
\begin{equation*}
r< \sqrt{1 + \left( \frac{\eta^{3} d^{2}}{16}\right)^{1/5}}= N_{3b}.
\end{equation*}
We apply Corollary \ref{Core3} with $A = \eta$ and $N = N_{3b}$. Since $b< (d/4)^{2/5} < 2.21\cdot 10^{23}$ we have $\omega(b) \leq 18$. Following the proof in \cite{EFF} we deduce that the number of quintuples is at most
\begin{equation}\label{quince3b}
4 \cdot 2^{18} \cdot 5 \cdot 4 \cdot 4N_{3b}\left( 3\pi^{-2} \log^{2} \eta + 1.3949\log \eta + 0.4107 + 3.253\eta^{-1/3}\right).
\end{equation}
We now try to minimise the maximum of (\ref{quince3a}) and (\ref{quince3b}) by choosing  $\eta$ judiciously. Indeed, at $\eta = 1.29\cdot10^{11}$, the number of quintuples is at most
\begin{equation}\label{quince3}
1.994\cdot 10^{25}.
\end{equation}

\subsection{Triples of the kind 2(iv) and triples of the third kind}
We use the bound $b< 2.66\cdot 10^{25}:= N$ as given in \cite{EFF}. One could improve this by examining Lemma 4.3 in \cite{EFF} combined with our improved Lemma \ref{lem:spec} --- we have not done this. We merely use Lemma \ref{bourbon} to show that the number of such quintuples is at most
\begin{equation}\label{quince4}
4 \left( 6 \pi^{-2} N \log N + 0.787 N - 0.3762 + 8.14 N^{2/3}\right) \leq 3.88\cdot 10^{27}.
\end{equation}

Insofar as triples of the third kind are concerned, we do not attempt to improve on the bound $b< 4.33\cdot 10^{23}:= N$ as obtained in Proposition 4.4 \cite{EFF}. We simply use  Lemma \ref{bourbon} to prove that the number of quintuples containing triples of the third kind is at most
\begin{equation}\label{3rdbound}
4 \left( 6 \pi^{-2} N \log N + 0.787 N - 0.3762 + 8.14 N^{2/3}\right) \leq 5.9\cdot 10^{25}.
\end{equation}
Using (\ref{quince1}), (\ref{quince2}), (\ref{quince3}), (\ref{quince4}) and (\ref{3rdbound}) proves Theorem \ref{Main}.


\section{$D(-1)$-quadruples}\label{sec:quad}
For $a<b<c<d$, a $D(-1)$-quadruple $\{a, b, c, d\}$ is a set with the property that the product of any two of its members is one more than a perfect square. It is conjectured that there are no $D(-1)$-quadruples: in \cite[Thm 1.3]{EFF} it was proved that there are at most $5\cdot 10^{60}$. We use Lemma \ref{Lem15} to offer a small improvement on this. Indeed, we merely insert the new bound of $\sum_{n=1}^{N}d(n^{2} + 1)$ into \cite[pp.\ 11-12]{EFF}. This proves
\begin{thm}
There are at most $3.01\cdot 10^{60}$ $D(-1)$-quadruples.
\end{thm}
We have not invested any more effort into improving this estimate; we note that in \cite{BBCM} it is announced that the bound can be reduced to $2.5\cdot 10^{60}$.

\section{Conclusion and remarks on further improvements}
We conclude with four suggestions to lower the bounds on $b$ and $d$ in various quadruples. If these bounds can be made sufficiently small, one could enumerate \textit{all possible} quadruples that could be extended to give quintuples. It would then be possible, as in \cite{PlattTrudgianQ1}, to refine this list multiple times, leaving only a `small' number of cases.
\label{sec:conch}
\begin{enumerate}
\item Eliminate more pairs and triples by generating more discards. 
\item Improve on the arguments of Matveev and Aleksentsev for sums of three logarithms. While it may be difficult to improve on the results for $n$ logarithms, it seems reasonable to predict some savings for our specialised application. To this end, see whether one can incorporate ideas from
\cite{Migkit}.
\item Make (\ref{eq:Ing}) and (\ref{pontifex}) explicit.
\item Improve on Lemma \ref{Vine}. Clearly there are no solutions for $b=1$.  It appears that the bound $2^{\omega(b) + 1}$ is obtained if and only if $b \equiv 0 \pmod 8$. Moreover, the number of solutions appears to be equal to $2^{\omega(b) -1}$ whenever $b \equiv 2, 6 \pmod 8$, and equal to $2^{\omega(b)}$ for $b\equiv 1, 3, 4, 5, 7 \pmod 8$. If this were true one should be able to improve on the Lemmas in \S \ref{sec:sums}.
\end{enumerate}

\bibliographystyle{plain}
\bibliography{themastercanada}

\end{document}